\documentclass[reqno,11pt]{amsart}
\usepackage[utf8]{inputenc}

\usepackage{amsmath}
\usepackage{mathtools}
\usepackage{amsfonts}
\usepackage{amssymb}
\usepackage{amsmath}
\usepackage{amsthm}
\usepackage{amssymb}
\usepackage{latexsym}
\usepackage{verbatim}
\usepackage{cleveref}
\usepackage{verbatim}
\usepackage[dvipsnames]{xcolor}
\usepackage{tikz-cd}
\usepackage{graphicx}

\theoremstyle{definition}
\newtheorem{definition}{Definition}[section]
\theoremstyle{definition}
\newtheorem{proposition}{Proposition}[section]
\newtheorem*{proposition*}{Proposition}

\newtheorem{lemma}{Lemma}[section]
\newtheorem{remark}{Remark}[section]
\newtheorem{assumption}{Assumption}[section]
\theoremstyle{corollary}
\newtheorem{corollary}{Corollary}[section]

\newcommand{\md}{\mathcal{D}}
\newcommand{\mcu}{\mathcal{U}}
\newcommand{\mcv}{\mathcal{V}}

\newcommand{\lb}{\bigtriangleup}
\newcommand{\intm}{\int_M}
\newcommand{\intb}{\int_{\partial M}}
\newcommand{\ol}[1]{\overline{#1}} 
\newcommand{\mx}{\mathcal{X}}
\newcommand{\lap}{\bigtriangleup}

\author{Alessio Marta}

\title[Heat and Schr\"{o}dinger equations with Wentzell conditions]
{Local well-posedness for the nonlinear heat and Schr\"{o}dinger equations with nonlinear Wentzell boundary conditions on time-dependent compact Riemannian manifolds}

\address{Dipartimento di Matematica\\ Universit{\`a} degli Studi di Milano\\ Via Cesare Saldini, 50 --  I-20133 Milano \\ Italy}

\email{alessio.marta@unimi.it}

\begin{document}

\maketitle
\begin{abstract}
We prove a local well-posedness result for the semilinear heat and Schr\"{o}dinger equations with subcritical nonlinearities posed on a time-dependent compact Riemannian manifold and supplied with a nonlinear dynamical boundary condition of Wentzell type.
\end{abstract}

\maketitle

\section{Introduction and main results}
Let $(M,g_t)$ be a smooth, oriented, compact Riemannian manifold with boundary $\partial M$, where $g_t$ is a time-dependent smooth Riemannian metric defined for $t \in [0,T]$, $T>0$. We assume that $g_t$ satisfies the following hypothesis.
\begin{assumption}\label{hypothesis:smooth_metric}
The family $\{g_t\}_{t \in [0,T]}$ depends smoothly on $t$.
\end{assumption}
Let $\nu$ be the outward-pointing unit normal of $\partial M$. We study the following related\footnote{In the language of quantum field theory, these two equations are related by means of a Wick rotation, with the Scr\"{o}dinger equation that can be seen as a heat equation with an imaginary time both in the interior and on the boundary.} nonlinear, non-autonomous problems:
\begin{itemize}
\item[1)] The nonlinear heat equation with nonlinear Wentzell boundary conditions
\begin{equation}\label{eq:heat_equation_standard_formulation}
\begin{cases}
\partial_t u = \lap y + \mathcal{N}(t,p,u) \quad & \textit{in} \quad M^\circ\ \\
\partial_t v = \lap v - \rho + \mathcal{N}_b(t,q,v) \quad & \textit{on} \quad \partial M\\
v = u|_{\partial M}, \rho = \nabla_\nu u|_{\partial M}
\end{cases}
\end{equation}
\item[2)] The nonlinear Scr\"{o}dinger equation with nonlinear Wentzell boundary conditions
\begin{equation}\label{eq:schrodinger_equation_standard_formulation}
\begin{cases}
\partial_t u = i \lap y + i \mathcal{N}(t,p,u) \quad & \textit{in} \quad M^\circ\ \\
\partial_t v = i \lap v - i \rho + i \mathcal{N}_b(t,q,v) \quad & \textit{on} \quad \partial M\\
v = u|_{\partial M}, \rho = \nabla_\nu u|_{\partial M}
\end{cases}
\end{equation}
where $p \in M^\circ$ and $q \in \partial M$.
\end{itemize}
In the following -- see \Cref{sec:evolution_equation_formulation} -- we shall understand the restrictions $u|_{\partial M}$ and $\nabla_\nu u|_{\partial M}$ as suitable traces in Sobolev spaces \cite{LionMagenes}.
\begin{remark}
Sometimes in the literature on the heat equation with dynamical boundary conditions, also conditions of the form $\lap v + \beta v - \rho = 0$, with $\beta$ typically a smooth function, are called dynamical conditions of Wentzell type \cite{Favini05,Arendt2003}. 
\end{remark}
The hypotheses we assume on the nonlinearities are standard. Let us denote with $\mathsf{C_M}$ and $\mathsf{C_{\partial M}}$ the critical exponents of the Sobolev embeddings for $H^1(M)$ and $H^1(\partial M)$ in the corresponding $L^2$ spaces, namely
\begin{equation}\label{eq:critical_exponents}
\mathsf{C_M} =
\begin{cases}
\dfrac{2n}{n-2} \ & \textit{if} \ n \geq 3\\
\infty \ & \textit{if} \ n = 2
\end{cases}
\quad
\mathsf{C_{\partial M}} =
\begin{cases}
\dfrac{2n-2}{n-3} \ & \textit{if} \ n \geq 4\\
\infty \ & \textit{if} \ n = 2,3
\end{cases}
\end{equation} 
Then we require that the nonlinear terms are satisfying the following conditions:
\begin{assumption}\label{hyp:nonlinearities_1}
Let $\alpha,\beta$ be two real numbers satisfying
\begin{equation}\label{eq:exponents_range}
\begin{cases}
1 \leq \alpha \leq \dfrac{\mathsf{C_M}}{2}\\
1 \leq \beta \leq \dfrac{\mathsf{C_{\partial M}}}{2}.
\end{cases}
\end{equation}
The nonlinear terms $\mathcal{N}$ and $\mathcal{N}_b$ seen as functions from $([0,T] \times M) \times \mathbb{R}$ and $([0,T] \times \partial M) \times \mathbb{R}$ to $\mathbb{C}$ are Carath\'eodory functions satisfying the following conditions:
\begin{itemize}
\item[a)] For every $u,v \in \mathbb{R}$ and $(t,p) \in [0,T] \times M$ and $(t,q) \in [0,T] \times \partial M$ there are $C,C_b \geq 0$ and $\alpha,\beta \geq 1$ such that 
\begin{equation}\label{eq:non_linearities_form}
\begin{split}
|\mathcal{N}(t,p,u)| \leq C ( 1 + |u|^{\alpha} )\\
|\mathcal{N}_b(t,q,v)| \leq C_b ( 1 + |v|^{\beta})
\end{split}
\end{equation}
\item[b)]  For every $u_1,u_2,v_1,v_2 \in \mathbb{R}$ and $(t,p) \in [0,T] \times M$ and $(t,q) \in [0,T] \times \partial M$ there are $K,K_b \geq 0$ and $\alpha,\beta \geq 1$ such that 
\begin{equation*}
|\mathcal{N}(t,p,u_1)-\mathcal{N}(t,p,u_2)| \leq K |u_1-u_2| \left( 1 + |u_1|^{\alpha-1} + |u_2|^{\alpha-1} \right)
\end{equation*}
\begin{equation*}
|\mathcal{N}_b(t,p,v_1)-\mathcal{N}_b(t,p,v_2)| \leq K_b |v_1-v_2| \left( 1 + |v_1|^{\beta-1} + |v_2|^{\beta-1} \right)
\end{equation*}
\end{itemize}
\end{assumption}
In the case $\mathsf{C_M}=\infty$ or $\mathsf{C_{\partial M}}=\infty$ we understand the inequalities in \Cref{eq:exponents_range} as $\alpha \geq 1$  and $\beta \geq 1$.
\begin{remark}
This assumption is met if we consider power type non-linearities of the form
\begin{equation*}
\begin{split}
\mathcal{N}(t,x,y) & = P(t,x) |y|^{\alpha-1}y\\
\mathcal{N}_b(t,x,z) & = P_b(t,x) |z|^{\beta-1}z
\end{split}
\end{equation*}
with $P \in \mathcal{C}^0([0,T]\times M)$, $P_b \in \mathcal{C}^0([0,T]\times \partial M)$.
\end{remark}

Wentzell boundary conditions were first introduced in \cite{Wentzell59} as the most general boundary conditions which restrict a second order diffusive elliptic operator on $\mathbb{S}^1$ and $\mathbb{S}^2$ to the infinitesimal generator of a positive contraction semigroup on the space of continuous functions over its domain. This kind of dynamic boundary conditions has been intensively studied in the last years for elliptic \cite{Fav00,Mugnolo06,Fav16,Hung18,Binz20}, parabolic \cite{Fa02,Arendt2003,VazVit08,Cav16,Gol20,cho22,giga22,chorfi23,Fila23,Ismailov23,Carvajal23,Mercado23} and hyperbolic \cite{Gal04,Favini05,Vitillaro16,Zahn17,DFJ18,JW20,JW21,DJM22,LiChan21,Li21,Vanspranghe20,Mehdi22,Vanspranghe22,Marta23} equations both in the linear and nonlinear case. In particular this type of boundary conditions have been considered in controllability problems both for the heat \cite{cho22,giga22,Carvajal23,chorfi23,Fila23} and the Schr\"{o}dinger equations \cite{Mercado23}.  Linear and nonlinear diffusion equations with dynamic boundary conditions find applications in models in which there are diffusion processes tangential to the boundary, such as heat transfer in the interface between a solid and a moving fluid \cite{Langer32} or population dynamics \cite{Farkas11}. In models describing biological processes \cite{mitra2006,Jianrong08} or stochastic processes with memory \cite{Carpio17}, non-autonomous evolution equations of the form $\partial_t y = \kappa f(t) \lap y$, with $\kappa \in \{1,i\}$ depending on the case and $f$ a smooth function of time, have been considered, therefore it is interesting to investigate also the case of time-dependent Laplace-Beltrami operators.  

Often in unbounded domains, let us say the positive half space of $\mathbb{R}^n$, the key for the study of the nonlinear Schr\"{o}dinger equation is given by a Strichartz estimate. In the case of bounded subsets, however, this type of estimates presents a loss in regularity compared to the unbounded case
\cite{Burq01,Anton05,Cav16,Macia2011,Ozsari15}. Furthermore, due to the inhomogeneous nature of the boundary condition this kind of estimate is difficult to obtain \cite{Cav16}. To overcome these hurdles, here we follow the approach employed in \cite{Marta23} to study the non-autonomous nonlinear wave equation with Wentzell boundary conditions, the main goal of this enquiry being to establish a local well-posedness result for the non-autonomous problems given in  \Cref{eq:heat_equation_standard_formulation,eq:schrodinger_equation_standard_formulation}. We accomplish this task recasting \Cref{eq:heat_equation_standard_formulation,eq:schrodinger_equation_standard_formulation} as first order evolution equations of the form
\begin{equation}\label{eq:evolution_formulation_simplified}
\dfrac{dX}{dt}(t)  = \kappa A(t)X(t) + \mathcal{N}(t,X)
\end{equation}
where the unknown $X=(u,v)$ is in $\mcu = L^2(M^\circ) \times L^2(\partial M)$, $A$ is a time-dependent operator expressing the linear part of the equation and $\mathcal{N}$ is a term containing the nonlinearities. In this expression the parameter $\kappa$ can be either $1$ or $i$ for heat and the Schr\"{o}dinger equation respectively. A fundamental physical requirement for the Schr\"{o}dinger equation is that the operator $A(t)$ must be self-adjoint for each $t \in [0,T]$, so we start our analysis studying the self-adjointness of the family of operators $\{A(t)\}_{t \in [0,T]}$. Making use of the notion of boundary triple, it turns out that for each time $t$ the operator $A(t)$ is indeed self-adjoint on its maximal domain (\Cref{prop:operator_a_selfadjoint}), a fact that shall be useful also for the study of the heat equation. The first step to study the well-posedness of these problems is to analyze the linearized equations with a time-independent metric. The result of this enquiry is that the evolution family associated with the linear heat equation is an analytic semigroup (\Cref{prop:analytic_semigroup}), while the one associated with the Schr\"{o}dinger problem is a unitary group (\Cref{prop:one-parameter-group}). We are then able to prove the following well-posedness proposition for the linear non-autonomous problems as per \Cref{eq:evolution_formulation_simplified} -- The definitions of classical and mild solutions are in \Cref{sec:analytic_preliminaries}.
\begin{proposition*}[\Cref{prop:low_regularity_non_auto_well_posedness}]
Let $F \in \mathcal{C}^1([0,T],\mcu)$ and suppose that $X(0) \in \mcu$. Then for $\kappa \in \{1,i\}$ the linear non-autonomous problems
\begin{equation*}
\begin{cases}
\dot{X} = \kappa A(t) X(t) + F(t), \quad t \in [0,T]\\
X(0) = X_0
\end{cases}
\end{equation*}
admit a unique mild solution $X(t) \in \mathcal{C}^0([0,T],\mcu)$ such that for every $t \in [0,T]$
\begin{equation*}
\| X(t) \|_\mcu \leq \|  X(0) \|_\mcu + \int_0^t \| F(s) \|_\mcu ds.
\end{equation*}
Furtermore, the solution solution $X(t) \in \mathcal{C}^1([0,T],\mcu) \cap \mathcal{C}^0([0,T],\mathcal{D}(A))$ is classical if $X(0) \in \md (A)$.
\end{proposition*}
Since $A$ is the generator of an analytic semigroup, in the case of the heat equation we can give a stronger version of this result -- The associated evolution family has a smoothing effect on initial data in $\mcu$ and the mild solution of the previous statement is actually in $\mathcal{C}^1([\varepsilon,T],\mcu) \cap \mathcal{C}^0([\varepsilon,T],\mathcal{D}(A))$ for every $\varepsilon \in (0,T)$. Next, we focus on the semilinear problem. Using a fixed-point argument we obtain the following conditional well-posedness result for initial data in a ball of $\mcu$.
\begin{proposition*}[\Cref{prop:nonlinear_well_posedness_conditional}]
Assume that Hypotheses \eqref{hypothesis:smooth_metric}, \eqref{hyp:nonlinearities_1} hold true. Let $\rho > 0$. Then there is $\tau = \tau(\rho)$ such that for every $X_0 \in \mcu$ with $\| X_0 \|_\mcu \leq \rho$, the problem in \Cref{eq:evolution_nonlinear_2} admits a unique mild solution $X \in C^0([0,\tau],\mcu)$ such that $\| X(t) \|_\infty \leq  \rho (M_0 + 1)$, where $M_0 = \sup_{t,s \in [0,T]} \| U(s,t) \|$.
\end{proposition*}
Proceeding as in \cite{Marta23}, we can then get the following (unconditional) local well-posedness result for mild solutions with initial data in $\mcu$.
\begin{proposition*}[\Cref{prop:nonlinear_well_posedness_unconditional}]
Assume Hypotheses \eqref{hypothesis:smooth_metric} and \eqref{hyp:nonlinearities_1}. Then the following assertions hold true:
\begin{itemize}
\item[a)] For each $u_0 \in \mcu$ there is a maximal mild solution $X(t) \in \mathcal{C}^0(I,\mcu)$ with $I$ either $[0,T]$ or $[0,t^+(X_0))$, where $t^+(X_0) \in [\tau,T]$.
\item[b)] If $t^+(X_0) < T$, then $\lim_{t\rightarrow t^+(X_0)^-} \|X(t) \|_\mcu=+\infty$.
\item[c)] For any $t^\star \in (0,t^+(X_0))$ there is a radius $\rho=\rho(X_0,t^\star)$ such that the map
$$ \overline{B(X_0,\rho)} \rightarrow \mathcal{C}^0([0,b],\mcu), \quad X_0 \mapsto X(t)$$
is Lipschitz continuous.
\end{itemize}
\end{proposition*}

The paper is organized as follows. In \Cref{sec:analytic_preliminaries} we introduce some basic notions and results concerning the tools we use to prove our results, namely boundary triples, evolutions semigroups and evolution families. In \Cref{sec:evolution_equation_formulation} we write the heat and the Schr\"{o}dinger equations \eqref{eq:heat_equation_standard_formulation} and \eqref{eq:schrodinger_equation_standard_formulation} in the form given in \Cref{eq:evolution_formulation_simplified}, introducing all the relevant functional spaces and proving that the operator $A$ is self-adjoint on its maximal domain. We then proceed to study the autonomous linearized problems first in the case of a time-independent metric in \Cref{subsec:time_independent_metric} and then with a time-dependent metric in \Cref{subsec:time_dependent_metric}. At last in \Cref{sec:nonlinear} we, focus on the semilinear equations, obtaining the sought local well-posedness result.
\section{Analytic preliminaries}\label{sec:analytic_preliminaries}
In this section we recall some basic facts on the notion of boundary triple and on the theory of evolution semigroups.

\subsection{Autonomous evolution equations}
Consider a Banach space $\mx$, $X_0 \in \mx$ and let $A : D(A) \subset \mx \rightarrow \mx$ be a linear operator. A function $X : \mathbb{R}_0^+ \rightarrow \mx$ is a classical solution to the problem
\begin{equation}\label{eq:acp}
\begin{cases}
\dot{X} = A X(t), \quad t \geq 0\\
X(0) = X_0 
\end{cases}
\end{equation}
if $X \in \mathcal{C}^1(\mathbb{R}_0^+;\mx)$, $X(t) \in D(A)$ and \Cref{eq:acp} is satisfied for every $t \geq 0$. If the operator $A$ is the generator of a one parameter semigroup $T(t)$, then the solution $X$ can be built making use of $T$.
\begin{proposition}[\cite{KlausRainer99}, Proposition 6.2]
Let $(A,D(A))$ be the generator of a strongly continuous semigroup $(T(t))_{t \geq 0}$. Then, for every $X_0 \in D(A)$, the function $X : t \mapsto X(t):=T(t)X_0$ is the unique classical solution of \Cref{eq:acp}.
\end{proposition}
The integral formulation of the problem in \Cref{eq:acp}
\begin{equation}\label{eq:acp_integral}
X(t) = A\int_0^t X(s) ds  + X_0
\end{equation}
may have sense also for initial data $X_0$ merely in $\mx$. We say that a continuous function $X : \mathbb{R}^+_0 \rightarrow \mx$ is a mild solution of \Cref{eq:acp} if $A\int_0^t X(s) ds \in D(A)$ and \eqref{eq:acp_integral} holds true for every $t > 0$.
Also in this case $X$ can be built by means of the  strongly continuous semigroup generated by $A$.
\begin{proposition}[\cite{KlausRainer99}, Proposition 6.4]
Let $(A,D(A))$ be the generator of a strongly continuous semigroup $(T(t))_{t \geq 0}$. Then, for every $X \in \mx$, the orbit map $X : t \mapsto X(t):=T(t)X_0$ is the unique mild solution of \Cref{eq:acp}.
\end{proposition}
\begin{remark}
Classical solutions are also solutions in the mild sense. 
\end{remark}
Let us now consider the inhomogeneous autonomous equation
\begin{equation}\label{eq:inhomogeneous_problem}
\begin{cases}
\dot{X} = A X(t) + F(t), \quad t \in [0,T]\\
X(0) = X_0
\end{cases}
\end{equation}
with $A$ the generator of a strongly continuous semigroup $(T(t))_{t \geq 0}$ and let $F: [0,T] \rightarrow \mx$ such that $\int_0^\delta \|F(s)\|_\mx ds \leq \infty$ for every $\delta \in (0,T]$.
The classical solution $X(t)\in C^1((0,T],\mx) \cap C^0([0,T],\mx)$ of this inhomogenous problem is given Duhamel's formula
\begin{equation}\label{eq:acpsol_source}
X(t) = T(t)X_0 + \int_0^t T(t-s)F(s) ds.
\end{equation}
As in the homogeneous case, the integral formulation may admit a mild solution in $\mathcal{C}^0([0,T],\mx)$ whenever \Cref{eq:acpsol_source} makes sense.
\begin{proposition}[\cite{Sch20}, Theorem 2.9]\label{prop:inhomogeneous_evolution_well_posedness_1}
Let $(A,D(A))$ be the generator of a strongly continuous semigroup $(T(t))_{t \geq 0}$. Let $X_0 \in \mathcal{D}(A)$. Assume that either $F \in \mathcal{C}^1([0,T],\mx)$ or that $F \in \mathcal{C}^0([0,T],\mathcal{D}(A))$. Then the function $X$ given by \Cref{eq:acpsol_source} is the unique classical solution of the inhomogeneous problem \eqref{eq:inhomogeneous_problem} on $[0,T]$. Furthermore, if the initial datum $X_0 \in \mx$ the solution $u \in C^0([0,T],\mx)$ given by \Cref{eq:acpsol_source} is the unique mild solution of the problem. 
\end{proposition}
If the source term $F$ satisfies additional regularity hypotheses, we can improve both the spatial and the temporal regularity of the solution.
\begin{proposition}[\cite{Can18}, Theorems 16,17]\label{prop:inhomogeneous_evolution_well_posedness_2}
Let $X_0 \in \mathcal{D}(A)$ and assume that $F$ is either in $L^2([0,T],\mx) \cap \mathcal{C}^0((0,T],\mx)$ or $H^1([0,T],\mx)$. Then the mild solution $X$ of \Cref{prop:inhomogeneous_evolution_well_posedness_1} is a strict solution, namely $X \in \mathcal{C}^1([0,T],\mx) \cap \mathcal{C}^0([0,T],\mathcal{D}(A))$.
\end{proposition}

\subsection{Analytic semigroups}

\begin{definition}
A linear operator $A : \mathcal{D}(A) \subset \mx \rightarrow \mx$ is sectorial if there exist three constants $\omega \in \mathbb{R}$, $\theta \in (\pi/2,\pi)$ and $M > 0$ such that the following two conditions are satisfied:
\begin{itemize}
\item[i)] The resolvent set $\rho(A)$ contains the set $$S_{\theta,\omega} = \{ \lambda \in \mathbb{C}, \ \lambda \neq \omega, \ \textit{such that} \ |arg(\lambda-\omega)|<\theta \}$$
\item[ii)] For every $\lambda \in S_{\theta,\omega}$ 
$$\| R(\lambda,A)\|_\mx \leq \dfrac{M}{|\lambda-\omega|} $$
\end{itemize}
\end{definition}
Thanks to the first resolvent identity, to satisfy condition i) it is enough to check if the resolvent set $\rho(A)$ contains a certain halfplane of $\mathbb{C}$, as staed in the following proposition.
\begin{proposition}[\cite{LLMP05}, Proposition 1.3.12]\label{prop:sectorial_simple_conditions}
Let $A : \mathcal{D}(A)  \subset \mx \rightarrow \mx$ be a linear operator such that $\rho(A)$ contains a halfplane $\{ \lambda \in \mathbb{C} \ : \ Re \ \lambda \geq \omega \}$, for some $\omega \geq 0$, and 
$\| \lambda R(\lambda,A) \|_\mx \leq M$, with $Re \ \lambda \geq \omega$, for some $M \geq 1$. Then $A$ is sectorial.
\end{proposition}
\begin{remark}\label{remark:selfadjoint_sectorial}
The condition on the resolvent set is satisfied in particular for self-adjoint operators with negative spectrum.
\end{remark}
The reason we are interested in sectorial operators with dense domains is that they are the generators of smoothing semigroups.\begin{proposition}[\cite{Sch20}, Theorem 2.25]
\label{prop:analytic_semigroup_properties}
Let $A$ be the generator of an analytic semigroup $T(t)$. Then for every $t > 0$ and $n>0$ we have $T(t) \mx \subseteq(\mathcal{D}(A^n))$, $\| A^n T(t) \| \leq M_n t^{-n}$ and $T(\cdot)\in \mathcal{C}^\infty(\mathbb{R}_+,\mathcal{B}(\mx))$.
\end{proposition}
In the case of inhomogeneous problems, we have the following result.
\begin{proposition}[\cite{Sch20}, Theorem 2.30]\label{prop:smoothing_analytic}
Let $X \in \mx$, $b>0$, $f\in \mathcal{C}^0([0,b],\mx)$ and $A-\omega \mathbb{I}$ be a densely defined sectorial operator for some $\omega \in \mathbb{R}$. Then the mild solution $X$ of the problem \eqref{eq:inhomogeneous_problem} satisfies the following assertions:
\begin{itemize}
\item[a)] $X$ belongs to the H\"{o}lder space $\mathcal{C}^\beta([\varepsilon,b],\mx)$ for every $\beta \in (0,1)$ and $\varepsilon \in (0,b)$.
\item[b)] If $f \in \mathcal{C}^\alpha([0,b],\mx)$ for some $\alpha > 0$, then $X$ is a classical solution on $[\varepsilon,b]$ for every $\varepsilon \in (0,b)$.
\end{itemize}
Furthermore, if $X_0 \in \mathcal{D}(A)$, then we can take $\varepsilon=0$ in the above assertions.
\end{proposition}
\subsection{Non-autonomous evolution equations}
Now we focus on non autonomous problems of the form
\begin{equation}\label{eq:non_autonomous_homogeneous}
\begin{cases}
\dot{X} = A(t)X(t) + F(t)\\
X(0) = X_0, \quad t \in [0,T]
\end{cases}
\end{equation}
with $F(t) : [0,T] \rightarrow \mx$, assuming that $\mathcal{D}(A(t))= \mathcal{D}(A(0))$ for every $t \in [0,T]$, $T \in \mathbb{R}$, is dense in $\mx$. Since the domain of $A(t)$ does not depend on time, we simply write $\mathcal{D}(A)$ to denote the domain of $A$ at a certain time. The treatment of this problem is different depending on the fact that $A(t)$ is the generator of an analytic semigroup or not. To discuss these two scenarios, we need to introduce the notion of evolution family first. 
\begin{definition}
Let $J \subseteq \mathbb{R}$ be a closed interval. A family of bounded operators $U(t,s)_{t \geq s}: \mx \rightarrow \mx$ with $s,t \in J$ is an evolution family if the following conditions are satisfied:
\begin{itemize}
\item[a)] $U(t,r)U(r,s) = U(t,s)$ for every $t\geq r \geq s$, $t,r,s \in J$.
\item[b)] For every $s \in J$ $U(s,s) = \mathbb{I}$.
\item[c)] The map $(t,s) \mapsto U(t,s)$  is strongly continuous.
\end{itemize}
\end{definition}
We say that $U(t,0)$ solves the non autonomous problem \Cref{eq:non_autonomous_homogeneous} if given an initial datum $X_0$ we have that
\begin{equation*}
X(t) = U(t,s)X_0 + \int_0^t U(t,\tau)f(\tau) d\tau
\end{equation*} 
for every $t \in [0,T]$. As for autonomous equation we say that the solution is strict if $X(t) \in \mathcal{C}^1([0,T],\mathcal{D}(A))$ and mild if $X(t) \in \mathcal{C}^0((0,T],\mx)$. 
\subsubsection{Parabolic problems}
If for each $t \in [0,T]$ the operator $A(t)$ is the generator of an analytic evolution semigroup, then the associated evolution family has a smoothing effect on initial data.
\begin{proposition}\label{prop:non_autonomous_homogeneous}
Let $A(t)$ generate an analytic $\mathcal{C}^0$ semigroup for every $t \geq 0$ and assume that $A(\cdot):[0,T] \rightarrow \mathcal{B}(\md (A),\mx)$ is H\"{o}lder continuous. Let $X_0 \in \mx$. Then there exists an evolution family $U(t,s)$ on $\mx$ solving problem \eqref{eq:non_autonomous_homogeneous} such that $U(0,0)=X_0$, $U(t,0) \mx \subseteq \mathcal{D}(A)$, $\partial_t U(t,0)=A(t)U(t,0)$ in $\mathcal{B}(\mx)$ and $\| A(t) U(t,s) \| \leq C/(t-s)$ for $0 \leq s < t \leq T$. The operator $U(t,s)$ can be built as the unique solution to the following integral equation:
\begin{equation*}
U(t,s) = e^{(t-s)A(s)} + \int_s^t U(t,\tau)[A(\tau)-A(s)]e^{(\tau-s)A(s)}d\tau
\end{equation*}
\end{proposition}
Consider now the non-autonomous inhomogeneous problem
\begin{equation}\label{eq:non_autonomous_inhhomogeneous}
\begin{cases}
\dot{X} = A(t)X(t) + F(t)\\
X(0) = X_0, \quad t \in [0,T]
\end{cases}
\end{equation}
The following result holds true \cite{AcBr87,Scha02Evo}.
\begin{proposition}\label{prop:non_autonomous_inhomogeneous_linear_solutions}
Suppose that $A$ satisfies the hypotheses of \Cref{prop:non_autonomous_homogeneous}, that $F \in \mathcal{C}^1([0,T],\mx)$ and that $X_0 \in \md (A)$. Then there is a unique classical solution $X(t) \in \mathcal{C}^1([0,T],\mx) \cap \mathcal{C}^0([0,T],\md (A))$. Furthermore, if $X_0 \in \mcu$, then \Cref{prop:non_autonomous_homogeneous} admits a unique mild solution $X(t) \in \mathcal{C}^0([0,T],\mx)$
\end{proposition}
By Duhamel's formula the solution can be written as \cite{Sch20}
\begin{equation*}
X(t) = U(t,s)X_0 + \int_0^t U(t,\tau)f(\tau) d\tau
\end{equation*} 

\subsubsection{Hyperbolic problems}\label{subsec:hyperbolic_problems}
If the operator $A(t)$ is merely the generator or a $\mathcal{C}^0$ semigroup which is not analytic for every $t \in [0,T]$, the evolution problem is called of hyperbolic type \cite{Kato70}. To make sure that the non-autonomous evolution problem in \Cref{eq:non_autonomous_inhhomogeneous} admits a solution, additional conditions must be met:
\begin{itemize}
\item[C1)] For every $t \in [0,T]$ the operator $A(t)$ is the generator of a $\mathcal{C}^0$ semigroup.
\item[C2)] the operator $B(s,t)=[\mathbb{I}-A(t)] \cdot [\mathbb{I}-A(s)]^{-1}$ is a uniformly bounded operator for every $t,s \in [0,T]$.
\item[C3)] $B(t,s)$ is of bounded variation in $t$, namely for evert $s$ there is $N \geq 0$ such that for every partition $0=t_0<t_1<\ldots<t_n=T$ of the interval $[0,T]$
\begin{equation*}
\sum_{j=1}^n \| B(t_j,s)-B(t_{j-1},s) \|_{\mathcal{B}(\mx)} \leq N
\end{equation*}
\item[C4)] $B(t,s)$ is weakly continuous in $t$ at least for some $s$.
\item[C5)] $B(t,s)$ is weakly differentiable in $t$ and the operator $\partial B(t,s) /\partial t$ is strongly continuous in $t$ at least for some $s$.
\end{itemize}
\begin{remark}
The first requirement of condition C2) is automatically satisfied when the domain of $A(t)$ is independent of $t$ by \cite[Lemma 2]{Kato53}
\end{remark}
Combining the results of \cite{Kato53} we have the following proposition.
\begin{proposition}\label{prop:well_posedness_hyperbolic_problems}
Assume that the conditions C1, C2, C3, C4, C5 hold true. Then there is an evolution family $U(t,s)$ defined $0 \leq s \leq t \leq T$ such that $U(t,s) \mathcal{D}(A) \subset \mathcal{D}(A)$ and the problem in \Cref{eq:non_autonomous_inhhomogeneous} admits a unique classical solution $X(t) = U(t,0) X_0$ for every $X_0 \in \mathcal{D}(A)$. Furthermore, if $X_0 \in \mx$ and $F(t) \in \mathcal{C}^0([0,T],\mx)$, \Cref{eq:non_autonomous_homogeneous} admits a unique mild solution $X(t) \in \mathcal{C}^0([0,T],\mx)$.
\end{proposition}

\subsection{Boundary triples}

\Cref{prop:analytic_semigroup_properties} applies in particular to self-adjoint operators with a negative spectrum. Given a symmetric operator, the existence of self-adjoint extensions can be established using the framework of boundary triples \cite{Grubb68}.

\begin{definition}\label{Def: boundary triples}
	Let $H$ be a separable Hilbert space over $\mathbb{C}$ and let $S:D(S)\subset H \rightarrow H$ be a closed, linear and symmetric operator.
	A {\em boundary triple} for the adjoint operator $S^*$ is a triple $(\mathsf{h},\gamma_0,\gamma_1)$, where $\mathsf{h}$ is a separable Hilbert space over $\mathbb{C}$ while $\gamma_0,\gamma_1:D(S^*) \rightarrow \mathsf{h}$ are two linear maps satisfying
	\begin{itemize}
		\item[1)] For every $f,f^\prime \in D(P^*)$ it holds
		\begin{equation}\label{eq:LagrangeId}
			(S^*f|f^\prime)_H-(f|S^*f^\prime)_H = (\gamma_1 f | \gamma_0 f^\prime)_{\mathsf{h}} - (\gamma_0 f | \gamma_1 f^\prime)_{\mathsf{h}}
		\end{equation}
		\item[2)] The map $\gamma:D(S^*) \rightarrow\mathsf{h}\times\mathsf{h}$ defined by $\gamma(f) = (\gamma_0 f, \gamma_1 f)$ is surjective.
	\end{itemize}
\end{definition}
Given a linear closed and symmetric operator, boundary triples allow to build its self-adjoint extensions thanks to the following proposition \cite{Mal92}.
\begin{proposition}\label{Prop: self-adjoint extensions via boundary triples}
	Let $S$ be a linear, closed and symmetric operator on $H$. Then an associated boundary triple $(\mathsf{h},\gamma_0,\gamma_1)$ exists if and only if $S^*$ has equal deficiency indices. In addition, if $\Theta:D(\Theta) \subseteq\mathsf{h}\rightarrow\mathsf{h}$ is a closed and densely defined linear operator, then $S_\Theta \doteq S^*|_{ker(\gamma_1-\Theta \gamma_0)}$ is a closed extension of $S$ with domain
	\begin{equation*}
		D(S_\Theta) \doteq \{ f \in D(S^*)\; |\; \gamma_0(f) \in D(\Theta), \; \textrm{and}\; \gamma_1(f) = \Theta\gamma_0(f) \}
	\end{equation*}
	The map $\Theta \mapsto S_\Theta$ is one-to-one and $S^*_\Theta = S_{\Theta^*}$. In other word there is a one-to-one correspondence between self-adjoint operators $\Theta$ on $\mathsf{h}$ and self-adjoint extensions of $S$.
\end{proposition}

\section{Evolution equation formulation}\label{sec:evolution_equation_formulation}

Let us consider the operator
\begin{equation}\label{eq:operator_A}
A(t)
= 
\begin{pmatrix}
\lap(t) & 0 \\
- \rho &  \lap_b(t)
\end{pmatrix}.
\end{equation}
Then, the linearized version of the heat and Schr\"{o}dinger equations as per \Cref{eq:heat_equation_standard_formulation,eq:schrodinger_equation_standard_formulation} can be formally written as
\begin{equation}\label{eq:heat_evolution}
\dfrac{\partial}{\partial t}
\begin{pmatrix}
u \\
v
\end{pmatrix}
=  A(t)
\begin{pmatrix}
u \\
v
\end{pmatrix}
+
F(t)
\end{equation}
\begin{center}
and
\end{center}
\begin{equation}\label{eq:schrodinger_evolution}
\dfrac{\partial}{\partial t}
\begin{pmatrix}
u \\
v
\end{pmatrix}
= 
i A(t)
\begin{pmatrix}
u \\
v
\end{pmatrix}
+F(t)
\end{equation}
respectively, with $v = \gamma_- u$ and $F$ a source term whose regularity we will discuss later. A fundamental requirement for the Schr\"odinger equation is that for every $t \in [0,T]$ the operator $iA(t)$ appearing in the right hand side of \Cref{eq:schrodinger_evolution} must be skew-adjoint \cite{Landau1981Quantum,Moretti13}, which is tantamount to say that the operator $A(t)$ is self-adjoint. As we will see, this last requirement will also yield the analiticity of the evolution semigroup for the heat equation. It turns out that the operator $A(t)$ on its maximal domain $\{ (u,v) \in H^2(M) \times H^2(\partial M) \ | \ v=\gamma_-u \}$ is indeed a self-adjoint operator for every $t \in [0,T]$. To prove this claim, we follow the same strategy employed in \cite{DDF18}. Let $\overline{t} \in [0,T]$ and consider the Laplace-Beltrami operator $\lap(\overline{t})$ in the interior of $M$, regarded as a densely defined operator $\lap(\overline{t}) : H^2_0(M) \rightarrow L^2(M,d \mu_{g(\overline{t})})$. By standard arguments $\lap(\overline{t})$ is a symmetric and closed operator on $L^2(M,d \mu_{g(\overline{t})})$, whose adjoint is defined on the maximal domain
\begin{equation*}
\mathcal{D}(\lap^*(t)) = \{ u \in L^2(M) \ | \ \lap^*(t) u \in L^2(M) \}
\end{equation*}
A way to characterize the self-adjoint extensions of $\lap(\overline{t})$ is to employ a boundary triple associated to the Laplace-Beltrami operator. Making use of \cite[Proposition 24]{DDF18} -- which generalizes a result originally proved in \cite{Grubb68} on bounded domains of $\mathbb{R}^n$ -- we immediately have the following proposition.
\begin{proposition}\label{prop:laplace-beltrami_boundary_triple}
Let $\lap^*$ be the adjoint of the Laplace-Beltrami operator on a compact Riemannian manifold $(M,g)$ with boundary. Let $\gamma_-$ and $\gamma_+$ be the Dirichlet and Neumann traces respectively, defined as
\begin{gather}\label{eq:dirichlet_trace}
\gamma_- \colon H^2(M)\ni u \mapsto \Gamma u \in L^2(\partial  M)\,,\\
\label{eq:neumann_trace}
\gamma_+ \colon H^2(M)\ni u\mapsto - \Gamma \nabla_\nu f\in L^2(\partial  M)\,,
\end{gather}
where $\Gamma : H^1(M) \rightarrow H^{1/2}(\partial M)$ is the Lions-Magenes trace and $\nu$ is the unit outward-pointing normal vector field of $\partial M$. Then $(L^2(\partial  M),\gamma_-,\gamma_+)$ is a boundary triple for $\Delta$.
\end{proposition}
\begin{proof}
Compact manifolds are trivially of bounded geometry, therefore we can apply \cite[Proposition 24]{DDF18} and the thesis follows.
\end{proof}
Now we introduce the functional spaces that we shall use, following \cite{Vitillaro16,Vanspranghe20,Marta23}. Set $\mcu = L^2(M) \times L^2(\partial M)$ and 
$\mathcal{V} = \{(u,v) \in H^1(M) \times H^1(\partial M) \ | \ v=\gamma_-u \}$. $\mathcal{V}$ is a Hilbert space defining for any $a=(a_1,a_2), \ b=(b_1,b_2) \in \mathcal{V}$ the inner product 
\begin{equation}\label{eq:inner_product_v}
\left( 
\begin{pmatrix}
a_1 \\
a_2
\end{pmatrix}
,
\begin{pmatrix}
b_1 \\
b_2
\end{pmatrix}
\right)_\mcv = \int_M   g(\nabla a_1, \nabla \ol{b_1}) dg + \int_{\partial M}   h(\nabla a_2, \nabla \ol{b_2}) dh
\end{equation}
We also consider the space $$\mathcal{W} = \mathcal{V} \cap \left( H^2(M) \times H^2(\partial M) \right)$$ which is a dense subspace of $L^2(M) \times L^2(\partial M)$.
Next, we prove that the operator $A(t)$ is self-adjoint on $\mathcal{W}$ for every $t$, following the arguments employed in \cite[Proposition 39]{DDF18}.
\begin{proposition}\label{prop:operator_a_selfadjoint}
The operator $A(t) : \mathcal{W} \rightarrow L^2(M) \times L^2(\partial M)$ defined in \Cref{eq:operator_A} is self-adjoint for every $t \in [0,T]$.
\end{proposition}
\begin{proof}
Let us consider the operator
\begin{equation}
P(t)
= 
\begin{pmatrix}
\lap^*(t) & 0 \\
- \gamma_+ &  \lap_b(t)
\end{pmatrix}
\end{equation}
By \Cref{prop:laplace-beltrami_boundary_triple}, we know that $(L^2(\partial M),\gamma_-,\gamma_+)$ is a boundary triple for $\lap^*(t)$, therefore the identity
\begin{equation}\label{eq:laplacian_green}
\begin{split}
(\lap^*(t)X,Y)_{L^2(M)}-(X,\lap^*(t)Y)_{L^2(M)} = \\ (\gamma_+ X , \gamma_- Y)_{L^2(\partial M)} - (\gamma_- X , \gamma_+ Y)_{L^2(\partial M)} 
\end{split}
\end{equation}
holds true for every $X,Y \in \mathcal{D}(\lap^*(t))$. Making use of this equality, a short computation unveils that
\begin{equation*}
(X,P(t)Y)_\mcu =(P(t)X,Y)_\mcu
\end{equation*}
namely the operator $P(t)$ is symmetric on its domain for every $t$, therefore we get the thesis if we show that $\mathcal{D}(P^*(t)) \subseteq \mathcal{D}(P(t))$. Let $X = (u_X,v_X) \in \mathcal{D}(P^*(t))$. Then for every $Y = (u_Y,v_Y) \in \mathcal{D}(P(t))$ and $\forall t \in [0,T]$ the map $Y \mapsto (X,P(t)Y)$ is bounded with
\begin{equation}\label{eq:inner_xay}
(X,P(t)Y)_{\mcu} = (u_X, \lap^*(t)u_Y)_{L^2(M)} + (v_X, \lap_b v_Y - \gamma_+ u_Y)_{L^2(\partial M)}
\end{equation}
from which we deduce that $v_X \in D(\lap_b^*(t))=D(\lap_b(t))=H^2(\partial M)$. Now it remains to prove that $u_X$ is in $\mathcal{D}(\lap^*(t))$ with $\gamma_-u_X = u_Y$. Combining \Cref{eq:inner_xay} and \Cref{eq:laplacian_green} we get
\begin{equation}\label{eq:P_star_and_P}
(P^*(t)X,Y)_{\mcu} = (PX,Y)_{\mcu} + (\gamma_- u_X -v_X,  \gamma_+ u_Y)_{L^2(\partial M)}.
\end{equation}
If $u_X \in ker(\gamma_+) \cap ker(\gamma_-)=H^2_0(M)$, then from the previous expression we find
\begin{equation*}
 (P^*(t)X,Y)_{\mcu} = (PX,Y)_{\mcu} = (\lb^*(t) u_X,u_Y)_{\mcu}
\end{equation*}  
which proves that the projection of $P^*$ on the first component is $\pi_1 P^*(t) X=\lb^* u_X$. Next, we take $u_Y \in Ker(\gamma_-)$. Considering that $\pi_1 P^*(t) X = \lap^*(t) u_X$, \Cref{eq:P_star_and_P} reads $0=(\gamma_- u_X -v_X,  \gamma_+ u_Y)_{L^2(\partial M)}$. Since the map $\gamma_+$ is surjective, we conclude that that $\gamma_- u_X = v_X$, proving that for every $t \in [0,T]$ the operator $P(t)$ is self-adjoint. The thesis follows noting that $P(t)$ coincides with $A(t)$ on its maximal domain $\mathcal{W}$.
\end{proof}
Stone's theorem on one-parameter unitary groups immediately yields the following result.
\begin{corollary}\label{prop:one-parameter-group}
For every $t \in [0,T]$ the operator $A(t)$ is the generator of a one-parameter unitary group $T(t)=e^{isA(t)}$ for every $s \in \mathbb{R}$.
\end{corollary}
The well-posedness of the linear non-autonomous Schr\"{o}dinger equation follows directly from the previous corollary. However, we do not state the result here, since in the next section we derive the well-posedness of both the heat and Schr\"{o}dinger problems at the same time using similar arguments.

\section{Linear equations: Time-independent metric}
\label{subsec:time_independent_metric}
To begin with, we prove that both $A(t^\star)$ and $iA(t^\star)$ are dissipative operators for every $t^\star \in [0,T]$. Since in this section the time $t^\star$ at which the operator $A$ is evaluated is fixed, for brevity we shall write $A$ instead of $A(t^\star)$. 
\begin{lemma}\label{lemma:dissipative}
Let $A$ be the operator given in \Cref{eq:operator_A}, with the Laplace-Beltrami operators $\lap$ and $\lap_b$ built out of a time-independent Riemannian metric. Then $Re(X,AX)_\mcu \leq 0$ and $Re(X,iAX)_\mcu = 0$ for very $X \in \mathcal{D}(A)$.
\end{lemma}
\begin{proof}
Let $X=(u,v) \in \mathcal{D}(A)$. Then integrating by parts we have
\begin{equation*}
\begin{split}
\left(X,AX \right)_\mcu & =
\left(
\begin{pmatrix}
u\\
v
\end{pmatrix},
\begin{pmatrix}
\lap u\\
- \rho u + \lap_b v
\end{pmatrix}
\right)_\mcu
 \\
& = -\int_M g(\nabla u, \overline{\nabla u}) d\mu_g - \intb h(\nabla v, \overline{\nabla v}) d \mu_h \leq 0
\end{split}
\end{equation*}
where $h = g|_{\partial M}$. As for the operator $iA$, the same computations unveil that
\begin{equation*}
\begin{split}
\left(X,iAX \right)_\mcu = -i \int_M g(\nabla u, \overline{\nabla u}) d\mu_g - i\intb h(\nabla v, \overline{\nabla v}) d \mu_h
\end{split}
\end{equation*}
from which we find that $Re \left(X,iAX \right)_\mcu = 0$.
\end{proof}

\begin{lemma}\label{lemma:surjectivity}
For every $\lambda>0$ the operators $A-\lambda \mathbb{I}: \mathcal{D}(A) \rightarrow \mcu$ and $iA-\lambda \mathbb{I}:\mathcal{D}(A) \rightarrow \mcu$ are surjective.
\end{lemma}
\begin{proof}
We prove in details the surjectivity of $A-\lambda \mathbb{I}$, the proof for $iA-\lambda \mathbb{I}$ being essentially the same. Let $\mathcal{F} = (F_1,F_2) \in \mathcal{U}$ and consider the equation
\begin{equation}\label{eq:surj}
AX - \lambda X = \mathcal{F}
\end{equation}
which we can write componentwise as
\begin{subequations}
\begin{align}
& \lap u  - \lambda u = F_1 \ \label{eq:surj1} \\
& -\gamma_+ u + \lap_b v - \lambda v = F_2 \ \label{eq:surj2}
\end{align}
\end{subequations}
Let $(a,b) \in \mcv$ and set $h = g|_{\partial M}$. Multiplying \Cref{eq:surj1} by $a$ and \Cref{eq:surj2} by $b$ and then integrating over $M$ and $\partial M$ respectively, yields the following variational formulation of the problem:
\begin{equation}\label{eq:surj_variational}
\begin{split}
\intm g(\nabla u, \nabla \ol{a}) dg + \intb h(\nabla v, \nabla \ol{b}) dh + \lambda \intm u \ol{a} \ dg + \lambda \intb v \ol{b} \ dh\\
 = -\intm F_1 \ol{a} \ dg - \intb F_2 \ol{b} \ dh
\end{split}
\end{equation}
for every $(a,b) \in \mcv$. Consider the bilinear form
\begin{equation}\label{eq:bilinear_form}
\begin{split}
\mathcal{B} \left(
\begin{pmatrix}
u\\
v
\end{pmatrix}
,
\begin{pmatrix}
a\\
b
\end{pmatrix}
\right)
= & \intm g(\nabla u, \nabla \ol{a}) dg  \intb h(\nabla v, \nabla \ol{b}) dh \\
 + & \lambda \intm u \ol{a} \ dg + \lambda \intb v \ol{b} \ dh
\end{split}
\end{equation}
appearing in \Cref{eq:surj_variational}. For $\lambda>0$ we have $\mathcal{B}(X,X) \leq (1 + \lambda)  \| X \|_\mcv^2$ and $\mathcal{B}(X,X) \geq \min\{ 1,\lambda \} \| X \|_\mcv^2$, therefore $\mathcal{B}$ is bounded and coercive and by Lax-Milgram theorem we conclude that the variational problem as per \Cref{eq:surj_variational} admits a unique solution in $\mcv$. Now we prove that the solution is actually in $\mathcal{W}$. Consider $(a,b) \in \mathcal{V}$ such that $a \in \mathcal{C}_0^\infty(M^\circ)$. Then, integrating by parts \Cref{eq:surj_variational} and considering that $F_1 \in L^2(M)$, we obtain that \Cref{eq:surj1} is satisfied almost everywhere, with $\lap u \in L^2(M^\circ)$. Since $u$ is a priori in $H^1(M)$, we know that $\gamma_+u \in H^{-1/2}(\partial M) $ and then by elliptic regularity theory, \Cref{eq:surj2} yields $v \in H^{3/2}(\partial M)$ with the estimate
\begin{equation*}
\| v \|_{H^{3/2}(\partial M)} \leq C \left( \| \gamma_+ u \|_{H^{-1/2}(\partial M)} + \| F_2 - \lambda v \|_{L^2(\partial M)} \right)
\end{equation*}
from which we conclude that $u \in H^2(M)$ and $\gamma_+u$ is actually in $H^{1/2}(\partial M)$. Considering again \Cref{eq:surj2} together with these data, elliptic regularity theory leads to the conclusion that $v \in H^2(\Gamma_1)$ and therefore $X \in \mathcal{D}(A)$, proving the thesis.\\

Finally we focus on the Scr\"{o}dinger operator $iA$. A quick computation shows that the variational formulation of $iA X + \lambda X = \mathcal{F}$ is
\begin{equation*}\label{eq:variational_schrodinger}
\begin{split}
i \intm g(\nabla u, \nabla \ol{a}) dg + i \intb h(\nabla v, \nabla \ol{b}) dh + \lambda \intm u \ol{a} \ dg + \lambda \intb v \ol{b} \ dh\\
 = \intm F_1 \ol{a} \ dg + \intb F_2 \ol{b} \ dh.
\end{split}
\end{equation*}
Since the associated bilinear form
\begin{equation}\label{eq:bilinear_form}
\begin{split}
\mathcal{B} \left(
\begin{pmatrix}
u\\
v
\end{pmatrix}
,
\begin{pmatrix}
a\\
b
\end{pmatrix}
\right)
= & \ i \intm g(\nabla u, \nabla \ol{a}) dg  + i \intb h(\nabla v, \nabla \ol{b}) dh \\
 + & \lambda \intm u \ol{a} \ dg + \lambda \intb v \ol{b} \ dh
\end{split}
\end{equation}
is still bounded and coercive -- in the sense of the complex version of Lax-Milgram theorem, see e.g. \cite{ErnGuermond04} -- the thesis follows by the same arguments employed for the heat operator $A$.
\end{proof}
The previous lemmas immediately yield the following proposition.
\begin{proposition}\label{prop:strongly_continous}
The operators $A$ and $iA$ generate strongly continuous semigroups $S_t = e^{At}$ and $T_t = e^{iAt}$ of linear contractions on $\mcu$.
\end{proposition}
\begin{proof}
$A$ and $iA$ are maximally dissipative operators with dense domain by \Cref{lemma:dissipative,lemma:surjectivity}, therefore the thesis follows applying Lumer-Phillips theorem.
\end{proof}
\begin{proposition}\label{prop:evolution_well_posedness}
Suppose that the source $\mathcal{F}=(F_1,F_2)$ is in $\mathcal{C}^1([0,T],\mcu)$. Then if  $X(0) \in \mcu$ the initial value problems for \Cref{eq:heat_evolution,eq:schrodinger_evolution} admit a unique mild solution $X(t) \in C^0([0,T],\mcu)$ satisfying
\begin{equation}
\|X(t)\|^2_\mcu \leq  \|X(0)\|^2_\mcu + \int_0^t \| \mathcal{F}(t) \|_{\mcu}^2
\end{equation}
In the case in which the initial data $X(0) \in \mathcal{D}(A)$, the problems in \Cref{eq:heat_evolution,eq:schrodinger_evolution} admit a unique classical solution $X(t) \in C^1([0,T],\mathcal{D}(A))$. 
\end{proposition}
\begin{proof}
The thesis immediately follows from \Cref{prop:strongly_continous} making use of \Cref{prop:inhomogeneous_evolution_well_posedness_1,prop:inhomogeneous_evolution_well_posedness_2}
\end{proof}
\begin{remark}
Consider a classical solution of the homogeneous Schr\"odinger problem as per \Cref{eq:schrodinger_evolution}. A short computation unveils that $$\partial_t \| X(t) \|_\mcu^2 = 2Re(X,AX)_\mcu$$
therefore \Cref{lemma:dissipative} yields that the $L^2$ mass of the interior-boundary system is conserved. As for the heat \cref{eq:heat_evolution}, a similar argument bring us to the conclusion that the $L^2$ norm is non-increasing in time.
\end{remark}
\begin{remark}
The well-posedness result for the Schr\"odinger problem keeps to hold true also if the non-autonomous problem is considered in $[-T,T]$, due to the fact that $iA$ is the generator of a one-parameter unitary group by \Cref{prop:one-parameter-group}.
\end{remark}
We conclude this section focusing on the generator $A$ of the heat equation.
\begin{proposition}\label{prop:analytic_semigroup}
The operator $A$ generates an analytic semigroup.
\end{proposition}
\begin{proof}
The thesis follows from \Cref{remark:selfadjoint_sectorial}, since $A$ is a self-adjoint operator (\Cref{prop:operator_a_selfadjoint}) whose spectrum is negative (\Cref{lemma:dissipative}).
\end{proof}
On account of \Cref{prop:smoothing_analytic} the evolution family of $A$ has a smoothing effect on initial data in $\mcu$, as stated in the following result.
\begin{corollary}\label{corollary:smoothing_effect}
Let $X_0 \in \mcu$ and $F \in \mathcal{C}^a([0,T],\mcu)$ for some $a>0$. Then \Cref{eq:heat_evolution} admits a classical solution $$X(t) \in \mathcal{C}^a([\varepsilon,T],\mathcal{D}(A)) \cap \mathcal{C}^{1+a}([\varepsilon,T],\mcu)$$ for every $\varepsilon \in (0,T)$.
\end{corollary}

\section{Linear equations : Time-dependent metric}\label{subsec:time_dependent_metric}

This section is devoted to prove the following well-posedness result for the heat and Scr\"{o}dinger equations.
\begin{proposition}\label{prop:low_regularity_non_auto_well_posedness}
Let $F \in \mathcal{C}^1([0,T],\mcu)$ and suppose that $X(0) \in \mcu$. Then for $\kappa \in \{1,i\}$ the linear non-autonomous problems
\begin{equation*}
\begin{cases}
\dot{X} = \kappa A(t) X(t) + F(t), \quad t \in [0,T]\\
X(0) = X_0
\end{cases}
\end{equation*}
admit a unique mild solution $X(t) \in \mathcal{C}^0([0,T],\mcu)$ such that for every $t \in [0,T]$
\begin{equation}\label{eq:bound_low_regularity_non_auto}
\| X(t) \|_\mcu \leq \|  X(0) \|_\mcu + \int_0^t \| F(s) \|_\mcu ds.
\end{equation}
Furthermore, the solution $X(t) \in \mathcal{C}^1([0,T],\mcu) \cap \mathcal{C}^0([0,T],\mathcal{D}(A))$ is classical if $X(0) \in \md (A)$.
\end{proposition}
The proof is different in the two cases, since in the case of the heat equation the non-autonomous problem is parabolic, while for the Scr\"{o}dinger case the problem is merely hyperbolic.
\begin{proof}
${ }$
\newline
\textbf{Heat equation}. The evolution operator $S(t)=exp(t A(t^\star))$ is analytic for every $t^\star \in [0,T]$ as a consequence of \Cref{prop:analytic_semigroup}, thus we can apply \Cref{prop:non_autonomous_inhomogeneous_linear_solutions} to get the thesis.\\
\textbf{Scr\"{o}dinger equation}. The thesis follows from \Cref{prop:well_posedness_hyperbolic_problems} once we check that its hypotheses are met. To this end we recast the operator
\begin{equation*}
B(t,s) = \left( iA(t) + \mathbb{I} \right) \left(iA(s) + \mathbb{I} \right)^{-1}
\end{equation*}
defined in \Cref{subsec:hyperbolic_problems} in the form
\begin{equation}\label{eq:operator_b_decomposition}
\begin{pmatrix}
i \lap(t) & 0\\
0 & i \lap_b(t)
\end{pmatrix}
\left(iA(s) + \mathbb{I} \right)^{-1} +
\begin{pmatrix}
1 & 0\\
-i\gamma_+ & 1 
\end{pmatrix}
\left(iA(s) + \mathbb{I} \right)^{-1}
\end{equation}
where we separated the terms depending on $t$ from those independent on $t$. We start our analysis on the latter. Since this term is non dependent on $t$ and $\left(iA(s) + \mathbb{I} \right)^{-1}$ is bounded by Hille-Yoshida theorem for every $s \in [0,T]$, conditions C2, C3, C4 and C5 are met. Next, we focus on the first term in the sum. 
By \Cref{hypothesis:smooth_metric} the operator
$$
\begin{pmatrix}
\lap(t) & 0\\
0 & \lap_b(t)
\end{pmatrix}
$$
depends on $t$ in a smooth way. Since $\left(iA(s) + \mathbb{I} \right)^{-1}$ is bounded for every $s \in [0,T]$ we conclude that C4 and C5 are met. For the same reason the first term in \Cref{eq:operator_b_decomposition} is uniformly bounded (C2) and of bounded variation (C3) for $s,t \in [0,T]$, where we also use that $\left(iA(s) + \mathbb{I} \right)^{-1} \mcu = \mathcal{D}(A)$ by \Cref{lemma:surjectivity}. We get the thesis once we note that condition C1 is met by \Cref{prop:evolution_well_posedness}.

In both cases Duhamel's principle yields the bound in \Cref{eq:bound_low_regularity_non_auto}.
\end{proof}
\begin{remark}
By \Cref{prop:non_autonomous_inhomogeneous_linear_solutions},   the non-autonomous heat equation retains the regularizing property of the autonomous case. In particular the mild solution $X(t) \in \mathcal{C}^0([0,T],\mcu)$ of the previous result corresponding to an initial datum in $\mcu$ is actually in $\mathcal{C}^1([\varepsilon,T],\mcu) \cap \mathcal{C}^0([\varepsilon,T],\mathcal{D}(A))$ for every $\varepsilon \in (0,T)$.
\end{remark}

\section{Local wellposedness for the nonlinear equations}
\label{sec:nonlinear}

Given the ingredients in \Cref{subsec:time_dependent_metric}, the local well-posedness for the nonlinear heat and Scr\"{o}dinger equations can be treated together, at least for mild solutions. Consider the problem
\begin{equation}\label{eq:evolution_nonlinear}
\dfrac{d}{dt}
\begin{pmatrix}
u \\ v
\end{pmatrix}
+ \kappa A(t)
\begin{pmatrix}
u \\ v
\end{pmatrix}
=
\begin{pmatrix}
\mathcal{N} \left(t,p_1,u(t)(p_1)\right) \\
\mathcal{N}_b \left(t,p_2,v(t)(p_2)\right)
\end{pmatrix}
\end{equation}
where the linear operator $A(t)$ was defined in \Cref{eq:operator_A} and $\kappa$ is either $1$ or $i$ for the heat and the Schr\"{o}dinger equation respectively. For every $t \in \mathbb{R}$, $p=(p_1,p_2) \in M \times \partial M$ and $X=(u,v) \in \mcu$ we define
\begin{equation}
\mathcal{G}(t,p,X)=
\begin{pmatrix}
\mathcal{N} \left(t,p_1,u_1(t)(p_1)\right)\\
\mathcal{N}_b \left(t,p_2,u_2(t)(p_2)\right)
\end{pmatrix}
\end{equation}
and
\begin{equation}
\mathcal{F}(t,X)(p)=\mathcal{G}(t,p,X)
\end{equation}
Thanks to \Cref{hyp:nonlinearities_1} for every $X \in \mcu$ the function $t \mapsto \mathcal{F}(t,X)$ is mapping continuously $[0,T]$ to $\mx$, therefore for every $T>0$ we have $\mathcal{F}(t,X) \in \mathcal{C}^0([0,T],\mcu)$.
With these data \Cref{eq:evolution_nonlinear} can be recast in the form
\begin{equation}\label{eq:evolution_nonlinear_2}
\dfrac{d}{dt}X + \kappa A(t)X = \mathcal{F}(t,X).
\end{equation}
Let $\| \cdot \|_\infty$ be the supremum norm on the space $\mathcal{C}^0\left( [0,\tau],\mcu \right)$ and define, for any real $\tau,\rho > 0 $, the space
$$B_{\tau,\rho} = \left\{ f \in \mathcal{C}^0\left( [0,\tau] , \mcu \right) \ \textit{such that} \ \|f\|_\infty \leq \rho  \right\}.$$
Then the pair $\left( B_{\tau,\rho},\| \cdot  \|_\infty \right)$ is a complete metric space by a standard argument. The reason why we introduced this space is the following lemma concerning the nonlinear term $\mathcal{F}$.
\begin{lemma}\label{lemma:nonlinear_lipschitz}
For every $\tau,\rho>0$ there is $L_{\tau,\rho}>0$ such that 
\begin{equation*}
\| \mathcal{F}(t,X)-\mathcal{F}(t,Y) \|_{\infty} \leq L_{\tau,\rho} \|X-Y\|_\infty
\end{equation*}
for every $X,Y \in B_{\tau,\rho}$.
\end{lemma}
\begin{proof}
Let $X=(u_X,v_X)$ and $Y=(u_Y,v_Y)$ be in $B_{\tau,\rho}$. Then by Assumption \eqref{hyp:nonlinearities_1} on the nonlinear terms and Sobolev embedding theorem, we have that
$\mathcal{N}(t,p,u_X) \in L^{\mathsf{C_M}}(M)$ and $\mathcal{N}_b(t,q,v_X) \in L^{\mathsf{C}_{\partial M}}(\partial M)$ for every $t \in [0,\tau], p \in M, q \in \partial M$, $\mathsf{C_M}$ and $\mathsf{C_{\partial M}}$ being the Sobolev embedding constants given in \eqref{eq:critical_exponents}. Now we prove that
\begin{equation*}
\mathcal{F}(t,X)-\mathcal{F}(t,Y) =
\begin{pmatrix}
\mathcal{N}(t,p,u_X) - \mathcal{N}(t,p,u_Y)  \\
\mathcal{N}_b(t,p,v_X) - \mathcal{N}_b(t,p,v_Y)  \\
\end{pmatrix}
\end{equation*}
is Lipschitz in the norm $\| \cdot \|_\infty$.
First we note that $\mathcal{F}(t,X)-\mathcal{F}(t,Y)$ is in $\mcu$ since $\mathsf{C_M},\mathsf{C_{\partial M}} \geq 2$. Let us now focus on the difference $\mathcal{N}(t,p,u_X) - \mathcal{N}(t,p,u_Y)$ in the interior. H\"{o}lder's inequality and \Cref{hyp:nonlinearities_1} yield
\begin{equation*}
\begin{split}
\| \mathcal{N}(t,p,u_X) - \mathcal{N}(t,p,u_Y)  \|_{L^2(M)} \leq \\
C_{\tau,\rho} \| u_x - u_Y \|_{L^{2a}(M)}^{1/2a}  \| 1 + |u_X|^{\alpha-1} + |u_Y|^{\alpha-1} \|_{L^{2b}(M)}^{1/2b} 
\end{split}
\end{equation*}
where $a$ and $b$ are conjugate exponents and $C_{\tau,\rho}$ is the constant appearing in \Cref{eq:non_linearities_form}. If we can choose $2b<\mathsf{C_M}/(\alpha-1)$ such that the conjugate exponent $2a$ is less than $\mathsf{C_M}$, applying Rellich-Kondrachov embedding theorem gives the inequality
\begin{equation}\label{eq:lipschitz_h1_interior}
\| \mathcal{N}(t,p,u_X) - \mathcal{N}(t,p,u_Y)  \|_{L^2(M)} \leq C(\rho,t) \| u_X(t) - u_Y(t) \|_{H^1(M)}
\end{equation}
with
\begin{equation*}
C(\tau,\rho) =  C_{\tau,\rho} \| 1 + |u_X|^{\alpha-1} + |u_Y|^{\alpha-1} \|_{L^{2b}(M)}^{1/2b}.
\end{equation*}
Now we prove that the aforementioned choice of $a$ and $b$ can be always done if \Cref{hyp:nonlinearities_1} is satisfied. A few computations yield that the inequality $2a < \mathsf{C_M}$ is tantamount to $b > \mathsf{C_M}/(\mathsf{C_M} - 2)$. Therefore we have two conditions on the exponent $b$: $2b<\mathsf{C_M}/(\alpha-1)$ and $b > \mathsf{C_M}/(\mathsf{C_M} - 2)$. Solving them with respect to $\alpha$ yields $\alpha \leq \mathsf{C_M}/2$, which is true by \Cref{hyp:nonlinearities_1}.

Proceeding along the same lines we get a similar inequality for the boundary nonlinear term, namely:
\begin{equation}\label{eq:lipschitz_h1_boundary}
\begin{split}
\| \mathcal{N}_b(t,p,v_X) - \mathcal{N}_b(t,p,v_Y)  \|_{L^2(\partial M)} \leq C_b(\rho,t) \| v_X(t) - v_Y(t) \|_{H^1(\partial M)}
\end{split}
\end{equation}
where $C_b(\rho,\tau)$ is non-decreasing in $\tau,\rho$.
Together \Cref{eq:lipschitz_h1_interior,eq:lipschitz_h1_boundary} leads to the bound
\begin{equation*}
\| \mathcal{F}(t,X)-\mathcal{F}(t,Y) \|_\mx \leq L_{\tau,\rho} \|X(t)-Y(t)\|_\mx
\end{equation*}
with $L_{\tau,\rho} = \max (C(\rho,\tau),C_b(\rho,\tau))$, from which the thesis follows taking the sup in $t \in [0, \tau]$.
\end{proof}
This lemma allows us to prove the following conditional well-posedness result by a fixed point argument.
\begin{proposition}\label{prop:nonlinear_well_posedness_conditional}
Assume that Hypotheses \eqref{hypothesis:smooth_metric}, \eqref{hyp:nonlinearities_1} hold true. Let $\rho > 0$. Then there is $\tau = \tau(\rho)$ such that for every $X_0 \in \mcu$ with $\| X_0 \|_\mcu \leq \rho$, the problem in \Cref{eq:evolution_nonlinear_2} admits a unique mild solution $X \in C^0([0,\tau],\mcu)$ such that $\| X(t) \|_\infty \leq  \rho (M_0 + 1)$, where $M_0 = \sup_{t,s \in [0,T]} \| U(s,t) \|$.
\end{proposition}
\begin{proof}
Let $\rho > 0$ and consider $X_0$ such that $\|X_0\|_\mx \leq \rho$. Set then $r = 1 + M_0 \rho$ and consider a time $0 < \tau \leq T$. For any $X \in B_{\tau,r}$ we introduce the map
\begin{equation}\label{eq:fixed_point_map}
\Phi_{X_0}(X)(t) = U(t,0) X_0 + \int_0^t U(t,s)F(s,X(s))ds
\end{equation}
We note that $X \in B_{\tau,\rho}$ is a solution of \Cref{eq:evolution_nonlinear_2} with initial datum $X_0$ if and only if $X$ is a fixed point of $\Phi_{X_0}$. The thesis then follows form Banach's theorem once its hypotheses are met. The first addend $U(t,0) X_0$ in \Cref{eq:fixed_point_map} is continuous by definition of evolution family. Furthermore, since \Cref{hyp:nonlinearities_1} is met, the function $F$ is continuos in $s$ and we find
\begin{equation}
\begin{split}
\left| \int_0^t U(t,s) F(X(t)) ds \right| \leq M_0 \int_0^t \left| F(s,X(t)) \right| \leq M \tau \|F\|_\infty.
\end{split}
\end{equation}
from which the dominated convergence theorem yields that $\Phi_{X_0}(X)$ is in $\mathcal{C}^0([0,\tau],\mx)$. Next we check that the map $\Phi_{X_0}(X)(\cdot)$ is Lipschitz continuous in $X$. For every $X,Y \in B_{\tau,\rho}$ we have
\begin{equation}
\begin{split}
\left\|\Phi_{X_0}(X)(t) - \Phi_{X_0}(Y)(t) \right\|_\mx & \leq M_0 \int_0^t \left\| \mathcal{F}(s,X(s))- \mathcal{F}(s,Y(s)) \right\|_\mx ds \\
& \leq M_0 \tau \left\| \mathcal{F}(s,X(s))- \mathcal{F}(s,Y(s)) \right\|_\mx
\end{split}
\end{equation}
which together with \Cref{lemma:nonlinear_lipschitz} yields
\begin{equation}\label{eq:evolution_lipschitz}
\left\|\Phi_{X_0}(X)(t) - \Phi_{X_0}(Y)(t) \right\|_\infty \leq \tau M_0 L_{\tau,\rho} \left\| X-Y \right\|_\infty.
\end{equation}
Since $L_{\tau,\rho}$ is non-decreasing in $\tau$, for $\tau$ small enough we have $\tau M_0 L_{\tau,\rho} < 1$ and the map $\Phi_{X_0}$ is a contraction on $B_{\tau,\rho}$. Therefore Banach's theorem guarantees that $\Phi_{X_0}$ admits a unique fixed point $X \in B_{\tau,\rho}$ -- The only solution to \Cref{eq:evolution_nonlinear_2} with initial datum $X_0$. At last we prove the bound the solution in the supremum norm. For any $t \in [0,\tau]$, using \Cref{lemma:nonlinear_lipschitz} together with the fact that $\mathcal{F}(s,0)=0$, we can write
\begin{equation}\label{eq:nonlinear_solution_bound}
\begin{split}
\|X(t)\|_\mx = \| \Phi_{X_0}(X)(t) \|_\mx \leq M_0 \| X_0 \|_\mx +  M_0 \int_0^t \|\mathcal{F}(s,X(s)) - \mathcal{F}(s,0) \|_\mx ds\\ \leq M_0 \rho + M_0 \tau L_{\tau,\rho} \rho
\end{split}
\end{equation}
If we choose 
\begin{equation}\label{eq:existence_time}
\tau = \min \left\{T, \dfrac{1}{2 M_0 L_{T,\rho}} \right\}
\end{equation}
we find a Lipschitz constant of $1/2$ and we can bound the norm of $X(t)$ independently of $\tau$ as follows.
\begin{equation*}
\| X(t) \|_\mx \leq \rho \left( 1 + M_0 \right)
\end{equation*}
\end{proof}
We can obtain an unconditional well-posedness result, gluing and shifting solutions in time. Since the proof is the same of \cite[Proposition 5.3]{Marta23}, we limit ourselves to state the final result.
\begin{proposition}\label{prop:nonlinear_well_posedness_unconditional}
Assume Hypotheses \eqref{hypothesis:smooth_metric} and \eqref{hyp:nonlinearities_1}. Then the following assertions hold true:
\begin{itemize}
\item[a)] For each $u_0 \in \mcu$ there is a maximal mild solution $X(t) \in \mathcal{C}^0(I,\mcu)$ with $I$ either $[0,T]$ or $[0,t^+(X_0))$, where $t^+(X_0) \in [\tau,T]$ and $\tau$ is as per \Cref{eq:existence_time}.
\item[b)] If $t^+(X_0) < T$, then $\lim_{t\rightarrow t^+(X_0)^-} \|X(t) \|_\mcu =+\infty$.
\item[c)] For any $t^\star \in (0,t^+(X_0))$ there is a radius $\rho=\rho(X_0,t^\star)$ such that the map
$$ \overline{B(X_0,\rho)} \rightarrow \mathcal{C}^0([0,b],\mcu), \quad X_0 \mapsto X(t)$$
is Lipschitz continuous.
\end{itemize}
\end{proposition}
\section*{acknowledgements}
This work was supported by a fellowship of the University of Milano.

\bibliographystyle{alpha}
\bibliography{biblio}
\end{document}